\documentclass{amsart}
\usepackage{amsmath, amstext, amsbsy, amssymb}

\hoffset \voffset \oddsidemargin=55pt \evensidemargin=55pt
\numberwithin{equation}{section}
\def\beq{\begin{eqnarray}}
\def\eeq{\end{eqnarray}}
\def\beqs{\begin{eqnarray*}}
\def\eeqs{\end{eqnarray*}}
\def\NN{{\mathbb N}}
\def\ZZ{{\mathbb Z}}
\def\dfrac{\displaystyle\frac}
\def\dim{{\hbox{\rm dim}}}
\def\dsum{\displaystyle\sum}

\def\a{\alpha}
\def\b{\beta}

\def\ep{\epsilon}
\def\dfrac{\displaystyle\frac}
\newfont{\df}{eufm10}

\def\r{\gamma}
\def\a{\alpha}

\def\CC{{\mathbb C}}

\def\mod{{\hbox{\rm mod}}}

\def\deg{\hbox{\rm deg}}

\def\id{\hbox{\rm id}}

\def\wg{\widetilde{\mathcal G}}

\title[vertex representations for $C_l^{(1)}$]
{Vertex Operator Representations of Type $C_l^{(1)}$ \\ and
Product-Sum Identities}
\thanks{$^\star$Corresponding author's email: xialimeng@ujs.edu.cn}
\thanks{Xia is supported by the  NNSF of China (Grant No. 11001110) and {\it Jiangsu Government Scholarship for Overseas Studies}. Hu is
supported in part by the NNSF of China (Grant No. 11271131), the
PCSIRT and the RFDP from the MOE of China, the National \& Shanghai
Leading Academic Discipline Projects (Project Number: B407).}
\author[L.M. Xia]{Li-meng   Xia$^{1,\star}$}
\author[N.H. Hu]{Naihong   Hu$^{2}$}

\date{}
\begin{document}
\maketitle \centerline{$^1$Faculty of Science, Jiangsu University}
\centerline{Zhenjiang 212013, Jiangsu Province, P.R. China}
\centerline{$^2$Department of Mathematics, East China Normal
University} \centerline{Dongchuan Road 500, Shanghai 200241, P.R.
China}

\def\abstractname{ABSTRACT}
\begin{abstract}

The purposes of this work are to construct a class of  homogeneous
vertex representations of $C_l^{(1)} \ (l\geq2)$, and to derive a
series of product-sum identities. These identities have fine
interpretation in number theory.

\end{abstract}

\newtheorem{theo}{Theorem}[section]
\newtheorem{theorem}[theo]{Theorem}
\newtheorem{defi}[theo]{Definition}
\newtheorem{lemma}[theo]{Lemma}
\newtheorem{coro}[theo]{Corollary}
\newtheorem{proposition}[theo]{Proposition}
\newtheorem{remark}[theo]{Remark}

\setcounter{section}{0}
\section{Introduction}

It is well known that there is a close relationship between
representations of affine Lie algebras and combinatorics. For
example,  the Jacobi triple product identity can be obtained as the Weyl-Kac denominator formula for the affine Lie algebra $\widehat{sl}_2$(\cite{K}). The famous Rogers¨CRamanujan identities can be realized from the character formula of certain level three representations \cite{LW1}. Like the Jacobi triple product identity, the quintuple product identity is also equivalent to the
Weyl-Kac denominator formula for the affine Lie algebra $A^{(2)}_2$.  In \cite{JX}, the following infinite product \beq
\prod_{n=1}^\infty\frac{1}{(1-q^{6n-1})(1-q^{6n-5})} \eeq is
expressed by a sum of two other infinite products in four different
ways.

I. Schur \cite{S} (see also \cite{A}) was probably the first person
who studied the partitions described by (1.1). He showed that the
number of partitions of $n$ into parts congruent to $\pm1(\mod\; 6)$
is equal to the number of partitions of $n$ into distinct parts
congruent to $\pm1(\mod\; 3)$ and is also equal to the number of
partitions of $n$ into parts that differ at least $3$ with added
condition that difference between multiples of $3$ is at least $6$.
His first result can be briefly described by \beq
\prod_{n=1}^\infty\frac{1}{(1-q^{6n-1})(1-q^{6n-5})}&=&\prod_{n=1}^\infty\frac{(1+q^{n})}{(1+q^{3n})}.
\eeq

Motivated by product-sum identity provided by \cite{JX}, we study a generalized product-sum relations of some special partitions. Our method uses the vertex
representations of affine Lie algebras of type $C_l^{(1)}$. For the
related topics, one can refer \cite{FLM}, \cite{G}, \cite{LH},
\cite{Mi}, \cite{XH} and references therein.

\begin{theorem}
For any odd $l\geq3$, the following product-sum identity holds:
\beqs
\prod_{n=1}^\infty\frac{(1+q^{n})}{(1+q^{ln})}&=&\sum_{s=0}^{\frac{l-1}2}q^{\frac{(l-2s)^2-1}{8}}\prod_{n\not\equiv \pm(s+1),0(\hbox{\rm mod }l+2)}\frac{1}{(1-q^{2n})(1-q^{ln})},
\eeqs
particularly, it covers the first result of \cite{JX} when $l=3$.
\end{theorem}

Our result in Theorem 1.1 implies the following partition theorem:

\begin{theorem}
Suppose that $l=2r+1\geq3$ is an odd number, $A_l(n)$ is the number
of partitions of $n$ into distinct parts without multiples of $l$,
and $B_{l,s}(n)$ is the number of partitions of $n$ into
$$2k_1+\cdots+2k_i+lr_1+\cdots+lr_j+\frac{(l-2s)^2-1}{8}$$ with constraints
$k_p,r_p\not\equiv\pm(s+1),0 (\mod\;l+2)$. Then for any positive integer $n$, we have
$$A_l(n)=B_{l,0}(n)+B_{l,1}(n)+\cdots+B_{l,r}(n).$$
\end{theorem}
\begin{proof}
Let $1+\sum_{n=1}^\infty A_na^n$ be the power series of $\prod_{n=1}^\infty\frac{(1+q^{n})}{(1+q^{ln})}$.  Because
\beqs \prod_{n=1}^\infty\frac{(1+q^{n})}{(1+q^{ln})}&=&\prod_{\hbox{\rm $n\geq 1$  is not a multiple of $l$}}(1+q^{n})\\
&=&\sum_{\hbox{\rm $n_1>n_2>\cdots>n_k\geq 1$}\atop\hbox{\rm $n_i$ is not a multiple of $l$,  $k\geq 0$}}q^{n_1+\cdots+n_k},\eeqs
Then $A_n$ is the number  of partitions of $n$ into distinct parts without multiples of $l$ and $A_n=A(n)$.

A similar argument on $B_{l,s}(n)$ shows that Theorem 1.1 is equivalent to   the relation
$$A_l(n)=B_{l,0}(n)+B_{l,1}(n)+\cdots+B_{l,r}(n),\quad  \hbox{\rm for all positive integer }n.$$
\end{proof}
For example,
$$\begin{array}{lllll}
A_5(15)=16&&B_{5,0}(15)=3&B_{5,1}(15)=7&B_{5,2}(15)=6\\
1+14&1+6+8&2(2\times3)+3&2(1\times7)+1&2(1\times5)+5(1)\\
2+13&2+4+9&2(2+4)+3&2(1\times4+3)+1&2(1\times3+2)+5(1)\\
3+12&2+6+7&2(3+3)+3&2(1+3\times 2)+1&2(1+2\times2)+5(1)\\
4+11&3+4+8&&2(1\times3+4)+1&2(5)+5(1)\\
6+9&1+2+3+9&&2(1+6)+1&5(1\times3)\\
7+8&1+2+4+8&&2(3+4)+1&5(1+2)\\
1+2+12&1+3+4+7&&2(1\times2)+5(1\times2)\\
1+3+11&2+3+4+6\\
\end{array}
$$
Table 1 lists the values of $A_5(n),B_{5,0}(n),B_{5,1}(n),B_{5,2}(n)$ for $n\leq15$.
$$\begin{array}{c|cccc}
 \hline n&A_5(n)&B_{5,0}(n)&B_{5,1}(n)&B_{5,2}(n)\\
\hline1&1&0&1&0\\
2&1&0&0&1\\
3&2&1&1&0\\
4&2&0&0&2\\
5&2&0&1&1\\
6&3&0&1&2\\
7&4&1&2&1\\
8&4&0&1&3\\
9&6&1&3&2\\
10&7&0&1&6\\
11&8&2&4&2\\
12&10&0&2&8\\
13&12&3&6&3\\
14&14&0&3&11\\
15&16&3&7&6\\
\hline
\end{array}$$
\centerline{\bf Table 1}

The above results will be proved by the irreducible decompositions of vertex module $V(P)=S(\widehat{H}^-)\otimes \CC[P]$ of $C_l^{(1)}$, where $1\otimes1$ has weight $\Lambda_0$. If we assume that $1\otimes 1$ has weight $\Lambda_1$, then our method also gives the following result:
\begin{theorem}
For any even $l\geq 2$, the following product-sum identity holds:
\beqs
\prod_{n=1}^\infty\frac{(1+q^{n-\frac12})^2}{(1+q^{n})(1+q^{ln})}&=&\frac{\prod_{n\geq1}(1-q^{(l(\frac{l+2}{2})(2n-1)})(1-q^{(l+2)(2n-1)})}{\prod_{(\frac{l}2+1)\not\;| n}{(1-q^{2n})(1-q^{ln})}}\\
&&+2\sum_{s=0}^{\frac{l}2-1}q^{\frac{(l-2s)^2}{8}}\prod_{n\not\equiv \pm(s+1),0(\hbox{\rm mod }l+2)}\frac{1}{(1-q^{2n})(1-q^{ln})},
\eeqs
or equivalently,
\beqs
\prod_{n=1}^\infty\frac{(1+q^{2n-1})^2}{(1+q^{2n})(1+q^{2ln})}&=&\frac{\prod_{n\geq1}(1-q^{(l(l+2)(2n-1)})(1-q^{2(l+2)(2n-1)})}{\prod_{ (\frac{l}2+1)\not\;| n}{(1-q^{4n})(1-q^{2ln})}}\\
&&+2\sum_{s=0}^{\frac{l}2-1}q^{\frac{(l-2s)^2}{4}}\prod_{n\not\equiv \pm(s+1),0(\hbox{\rm mod }l+2)}\frac{1}{(1-q^{4n})(1-q^{2ln})}.
\eeqs
\end{theorem}

Throughout the paper, we let $\CC, \ZZ$ present the set of complex numbers and the set of integers, respectively.

\section{Affine Lie algebra of type $C_l^{(1)}$}
\subsection{} Let $\dot{\mathcal G}$ be a
finite-dimensional simple Lie algebra of type $C_l$,
$A=\CC[t^{\pm1}]$ the ring of Laurent polynomials in variable $t$.
Then the affine Lie algebra of
type $C_l^{(1)}$ is the vector space
$$\widetilde{\mathcal G}=\dot{\mathcal G}\otimes A\; \oplus \;\CC c\oplus\CC d,$$
with Lie bracket: \beq{[\,x\otimes t^m, y\otimes t^n\,]}&=&[\,x,
y\,]\otimes t^{m+n}+m(x\,|\,y)\delta_{m+n,0}c,\\
{[\,c, {\mathcal G}\,]}&=&0,\\
{[\,d, x\otimes t^m\,]}&=&mx\otimes t^{m},\eeq where $x, \,y\in \dot{\mathcal G}$, $m,n\in\ZZ$ and $(\cdot\,|\,\cdot)$ is a nondegenerate
invariant normalized  symmetric bilinear form on $\dot{\mathcal
G}$.

\subsection{} Suppose that $\dot H$ is a
Cartan subalgebra of $\dot{\mathcal G}$, and $\dot H^*$ the dual
space of $\dot H$. Then there exists an inner product
$(\cdot\,|\,\cdot )|_{\dot H^*_{\mathbf R}}$ and an orthogonal normal
basis $\{e_1, e_2, \cdots, e_l\}$ in Euclidian space $\dot
H^*_{\mathbf R}$ such that the simple root system
$$\Pi=\left\{\,\alpha_1=\frac{1}{\sqrt{2}}(e_1-e_2),\, \cdots, \,\alpha_{l-1}=\frac{1}{\sqrt{2}}(e_{l-1}-e_l),
\,\alpha_l=\sqrt{2}\, e_l\, \right\},$$ the short root system
$$\dot\Delta_S =\left.\left\{\, \pm\frac{1}{\sqrt{2}}(e_i-e_j), \;\pm\frac{1}{\sqrt{2}}(e_i+e_j)\,\right| \,
1\leq i<j\leq l\,\right\},$$ where
$\frac{1}{\sqrt{2}}(e_i-e_j)=\alpha_i+\cdots+\alpha_{j-1}$ for
$1\le i<j\le l$,
$\frac{1}{\sqrt{2}}(e_i+e_l)=\alpha_i+\cdots+\alpha_l$ for $1\le
i<l$,
$\frac{1}{\sqrt{2}}(e_i+e_j)=\alpha_i+\cdots+\alpha_{j-1}+2\alpha_j+\cdots+2\alpha_{l-1}+\alpha_l$
for $1\le i<j<l$;

\noindent
and the long root system
$$\dot\Delta_L =\left.\left\{\,\pm\sqrt{2}\,e_i\,\right| \, 1\leq i\leq l\,\right\},$$
where $\sqrt{2}\,e_i=2\alpha_i+\cdots+2\alpha_{l-1}+\alpha_l$
for $1\le i<l$.

Then the root lattice is
$$Q=\bigoplus_{i=1}^l \ZZ\,\alpha_i$$
and $(\alpha_i | \alpha_i)=1$ ($1\leq i\leq l-1$), and $(\alpha_l
| \alpha_l)=2$.

Let $\gamma: {\dot H} \longrightarrow {\dot H}^*$ be the linear
isomorphism such that
$$\alpha_i(\gamma^{-1}(\alpha_j))=(\alpha_i | \alpha_j), \quad i, \,j=1,\cdots, l,$$
and
$$\gamma(\alpha_i^\vee)=2\alpha_i \ (i=1, \cdots, l-1), \quad \gamma(\alpha_l^\vee)=\alpha_l.$$
Then we have  $(\alpha_i^\vee |
\alpha_j^\vee)=(\gamma(\alpha_i^\vee)|\gamma(\alpha_j^\vee))$. As
usual, we identify $\dot H$ with $\dot H^*$ via $\r$,
i.e., $\a^{\vee}={2\a\over {(\a\,|\, \a)}}$.

For any weight $\Lambda\in (\dot H\oplus\CC c\oplus\CC d)^*$, let $L(\Lambda)$ denote the irreducible highest weight $\widetilde{\mathcal G}$-module with highest weight $\Lambda$.

\subsection{}
Define a $2$-cocycle $\epsilon_0 : Q\times Q
\longrightarrow\{\pm1\}$ by
$$\epsilon_0(a+b, c)=\epsilon_0(a, c)\,\epsilon_0(b, c), \quad
\epsilon_0(a, b+c)=\epsilon_0(a, b)\,\epsilon_0(a, c),\qquad a,
\,b, \,c\in L,$$ and
$$\epsilon_0(\alpha_i,
\alpha_j)=\left\{\begin{array}{rl}
-1,&i=j+1,\\
1,&\hbox{\it other pairs}~~(i,j),\\
\end{array}\right.$$

Let $P=\bigoplus_{i=1-1}^l\ZZ\,\alpha_i\oplus \frac12\ZZ\,\alpha_l$.
Extend $\epsilon_0$ to $Q\times P$ with
$$\epsilon_0(\alpha_i,\frac{1}{2}\alpha_l)=1.$$

\subsection{} For
$\alpha=\sum_{i=1}^{l}k_i\alpha_i\in\dot\Delta\cup\{0\}$, define
maps $p:\dot\Delta\cup\{0\}\rightarrow \dot\Delta_S\cup\{0\}$ and
$s:\dot\Delta\cup\{0\}\rightarrow \dot Q_L$ by:
$$p\left(\sum_{i=1}^lk_i\alpha_i\right)
=\sum_{i=1}^{l-1}\rho(k_i)\alpha_i,\qquad
s\left(\sum_{i=1}^lk_i\alpha_i\right)=\sum_{i=1}^{l-1}\bigl(k_i-\rho(k_i)\bigr)\alpha_i,$$
where $\dot Q_L=\hbox{Span}_{\ZZ}\,\dot\Delta_L$ and $\rho(k_i)\in
\{0,\,1\}$ such that $\rho(k_i)\equiv k_i~ \mod\; 2$. It is
straightforward to check the following statements.

\begin{lemma} $(\text{\rm i})$  $p\,(\dot\Delta_L\cup\{\,0\,\})=0$, and
$p(-\a)=p(\a)$ for any $\a\in\dot\Delta_S$.

$(\text{\rm ii})$ Suppose that $\a$, $\b$, $\a+\b\in \dot\Delta$,
then we have:

$(1)$ if $\a\in \dot\Delta_L$, then $(\a\mid \b)=-1$,
$p(\a+\b)=p(\b)$, $s(\a+\b)=s(\a)+s(\b)$;

$(2)$ if $\a$, $\b\in\dot\Delta_S$, $\a+\b\in\dot\Delta_L$, then
$(\a\mid \b)=0$, $p(\a)=p(\b)$, $s(\a+\b)-s(\a)-s(\b)=2p(\a)$;

$(3)$ if $\a$, $\b$, $\a+\b\in\dot\Delta_S$, then $(\a\mid
\b)=-\frac{1}{2}$ and $|\,(p(\a)\mid p(\b))\,|=\frac{1}{2}$;
Moreover,

\quad $(a)$  $(p(\a)\mid p(\b))=\frac{1}{2}$, then
$p(\a+\b)=p(\a)-p(\b)$, $s(\a+\b)-s(\a)-s(\b)=2p(\b)$, or
$p(\a+\b)=-p(\a)+p(\b)$, $s(\a+\b)-s(\a)-s(\b)=2p(\a)$;

\quad $(b)$  if $(p(\a)\mid p(\b))=-\frac{1}{2}$, then
$p(\a+\b)=p(\a)+p(\b)$, $s(\a+\b)=s(\a)+s(\b)$.

 $(\text{\rm iii})$  For any $\a\in\dot\Delta$, we have:

$(1)$ \ $s(\alpha)\in\{\pm{\sqrt{2}}(e_i-e_l) \mid 1\le i\le
l\}\subset\dot Q_L$; \

$(2)$ \ $p(\alpha)\in\{\frac{1}{\sqrt{2}}(e_i-e_j)\mid 1\leq i\leq
j\leq l\}\subset \dot\Delta_S\cup\{0\}$; \

$(3)$ \ $s(\a)+s(-\a)=-2p(\a)\in\dot Q_L$; \

$(4)$ \ $\a\pm p(\a)\in\dot Q_L$.
\end{lemma}

\subsection{} Define a map $f:\; Q\times
Q\rightarrow\{\pm1\}$ by
$$f(\alpha, \beta)
=(-1)^{(s(\alpha)\mid\b)+(p(\a)\mid p(\b)+p(\a+\b))}.
$$
Set $\epsilon=\epsilon_0\circ f$, then $\ep: Q\times
Q\longrightarrow\{\pm1\}$ is still a $2$-cocycle, which has the
property (ii) in the following Lemma.

\begin{lemma} $(\text{\rm i})$  For $\alpha, \beta\in\dot\Delta$, we have
$$
\epsilon_0(\alpha, \beta)=(-1)^{(\alpha | \beta)+(p(\alpha) |
p(\beta))+(s(\alpha) | \beta)+(s(\beta) | \alpha)}
\,\cdot\,\epsilon_0(\beta, \alpha).
$$
$(\text{\rm ii})$  For $\a,\;\b,\;\a+\b\in\dot\Delta$, we have
$\epsilon(\alpha, \beta)=-\epsilon(\beta, \alpha).$
\end{lemma}

\subsection{} We
have
\begin{proposition} The affine Lie algebra $\wg$ of type $C_l^{(1)}$ has a
system of generators $$\{\alpha_i^\vee\otimes t^{n}, e_\alpha\otimes
t^{n}\mid 1\le i\le l, n\in\ZZ\}$$ and $c, d$ with relations
$$\begin{array}{rcl}
{\bigl[\,\alpha_i^\vee\otimes t^{m},\; \alpha_j^\vee\otimes
t^{n}\,\bigr]}
&=&m(\alpha_i^\vee\,|\,\alpha_j^\vee)\delta_{m+n,0}c,\\
{\bigl[\,\alpha_i^\vee\otimes t^{m},\; e_\alpha\otimes
t^{n}\,\bigr]}
&=&\alpha(\alpha_i^\vee)\,e_\alpha \otimes t^{m+n},\\
{\bigl[\,e_\alpha\otimes t^{m},\; e_{-\alpha}\otimes
t^{n}\,\bigr]} &=&\epsilon(\alpha,
-\alpha)\,\dfrac{2}{(\alpha\,|\, \alpha)}\,
{\bigl[\,\gamma^{-1}(\alpha)\otimes t^{m+n}+m\delta_{m+n,0}c\bigr]},\\
{\bigl[\,e_\alpha\otimes t^{m},\; e_{\beta}\otimes
t^{n}\,\bigr]}
&=&\epsilon(\alpha,\beta)\,\bigl(\,1+\delta_{1,\,(p(\a)\,|\,p(\b))}\bigr)\,e_{\a+\b}\otimes
t^{m+n}, \quad \forall \;
\alpha, \;\beta, \;\alpha+\beta\in\dot\Delta,\\
{\bigl[\,e_\alpha\otimes t^{m},\; e_{\beta}\otimes
t^{n}\,\bigr]}&=&0,\quad \forall \;\alpha,\;\beta\in\dot\Delta,
 \;\alpha+\beta\not\in\dot\Delta\cup\{\,0\,\},
 \end{array}$$
where $\gamma$ is the canonical linear space isomorphism from
$\dot H$ to $\dot H^*$.
\end{proposition}

\section{Vertex construction of Lie algebra of type $C_l^{(1)}$}

\subsection{} Let $H(m)\;(m\in\ZZ)$ be an isomorphic copy of $\dot H$.
Set $\dot H_S:=\hbox{Span}_\CC\{\,\a_i\mid 1\le i\le
l{-}1\,\}$ and $H_S(n-\frac12)\;(n\in\ZZ)$ is an
isomorphic copy of $\dot H_S$.

Define a Lie algebra
$$\widehat{H}=\bigoplus_{m\in\ZZ}H(m)\oplus
\bigoplus_{n\in\ZZ}H_S(n-\frac{1}{2})\oplus \CC c,$$ with
Lie bracket
$$\begin{array}{rcl}
{\bigl[\,\widetilde{H}, c\,\bigr]}&=&0,\\
{\left[\,a(m), b(n)\,\right]}&=&m\,(a | b)\,\delta_{m,-n}c.
\end{array}$$

 Let
$$\widehat{H}^-=\bigoplus_{m\in\ZZ_-}H(m)\oplus \bigoplus_{n\in\ZZ_-}H_S(n+\frac{1}{2}),$$
and let $S(\widehat{H}^-)$ be the symmetric algebra generated by
$\widehat{H}^-$. Then $S(\widehat{H}^-)$ is an
$\widehat{H}$-module with the action
$$c\cdot v=v,\quad a(m)\cdot v=a(m)v, \quad\forall\; m<0,$$
and
$$a(m)\cdot b(n)=m\,(a, b)\,\delta_{m+n,0},  \quad \forall\; m\ge0,\; n<0,$$
where $a,\; b\in H$, $m,\;n\in\frac{1}{2}\ZZ$.

\subsection{} We form a group algebra $\CC[P]$
with base elements $e^{h} (h\in P)$, and the multiplication
$$e^{h_1}e^{h_2}=e^{h_1+h_2},\quad\forall \;h_1,\; h_2\in P.$$

Set
$$V(P):=S(\widehat{H}^-)\otimes \CC[P]$$
and extend the action of $\widehat{H}$ to space $V(P)$ by
$$a(m)\cdot(v\otimes e^{r})=(a(m)\cdot v)\otimes e^{r},\quad \forall\;
m \in\frac{1}{2}\ZZ^*;$$ and define
$$a(0)\cdot(v\otimes e^{r})=(a | r)\,v\otimes e^{r},$$
which makes $V(P)$ into a $\widehat H$-module.

\subsection{} For $r\in P, \;\a\in Q$, define $\CC$-linear
operators as
$$\begin{array}{rcl}
e^{\alpha}\cdot\left(v\otimes e^{r}\right)&=&v\otimes e^{\alpha+r},\\
z^{\alpha}\cdot \left(v\otimes e^{r}\right)&=&z^{(\alpha | r)}\,v\otimes e^{r},\\
\epsilon_\alpha\cdot\left(v\otimes e^{r}\right)&=&(-1)^{(s(\alpha) | r)}\,\epsilon_0(\alpha, r)\,v\otimes e^{r},\\
a(z)&=&\sum_{j\in\ZZ}a(j)\,z^{-2j},\\
E^\pm(\alpha, z)\cdot\left(v\otimes e^{r}\right)&=&
\left(\,\exp\bigl(\mp\dsum_{n=1}^\infty\dfrac{1}{n}z^{\mp 2n}\alpha(\pm n)\bigr)\cdot v\,\right)\otimes e^r,\\
F^\pm(\alpha, z)\cdot\left(v\otimes e^{r}\right)&=&
\left(\,\exp\bigl(\mp\dsum_{n=0}^\infty\dfrac{2}{2n{+}1}z^{\mp(2n{+}1)}\alpha\bigl(\pm\dfrac{2n{+}1}{2}\bigr)\bigr)\cdot
v\,\right)\otimes e^r.
\end{array}$$
Then $a(z),\; E^\pm(\alpha, z), \;F^\pm(\alpha, z)\in (\hbox{End}
V(P))[[z, z^{-1}]]$.

As usual, we shall adopt the notation of
normal ordering product
$$:a(i)b(j):=\left\{\begin{array}{ll}
a(i)b(j),& \hbox{if} ~~i\leq j,\\
b(j)a(i),& \hbox{if} ~~j<i,\\
\end{array}\right.$$
where $a,b\in L$ and $i,j\in\frac{1}{2}\ZZ.$

\subsection{} Let  $\widetilde{V}(P)$ be the formal
completion of $V(P)=S(\widehat{H}^-)\otimes \CC[P]$. We give some
vertex operators on $\widetilde{V}(P)$:

(1) \ For $\alpha\in \dot \Delta\cup\{\,0\,\}$, set
$$Y(\alpha, z)=E^-(\alpha,z)E^+(\alpha,z)
F^-(p(\alpha),z)F^+(p(\alpha),z),$$
$$Z^\epsilon(\alpha,z)=z^{(\a |
\a)}e^{\a}z^{2\a}\epsilon_\a,$$
$$X^\epsilon(\alpha, z):=Y(\alpha, z)\otimes Z^\epsilon(\a,z).$$

(2) \ For $\alpha,\;\beta\in \dot \Delta$, define
$$X^\epsilon(\a, \b, z, w) =\;:Y(\a, z)Y(\b,w):\otimes\;Z^\ep(\a+\b, w).$$

\subsection{} The Laurent series of operators
$X^\ep(\a, z)$ is denoted by $$X^\ep(\a, z)=\sum_{k=-\infty}^\infty
X^\ep_{\frac{k}{2}}(\a)\,z^{-k}.$$ Then $\forall\; k\in\ZZ$,
$X^\ep_{\frac{k}{2}}(\a)$ is an operator on $V(P)$. Note that
$X^\ep_n(\a)$ acts as an operator on $V(P)$ in the following way:
$$
X^\epsilon_n(\alpha)\cdot(v\otimes
e^r)=\,\ep(\a,r)\,Y_{n+\frac{1}{2}(\a | \a)+(\a|
r)}(\a)\,(v)\otimes e^{\a+r}, \quad
\forall\;v\otimes e^r\in V(P).$$

\subsection{} For $v=a_1(-n_1)a_2(-n_2)\cdots
a_p(-n_p)\otimes e^r\in V(L)$, define the degree action of $d$ on
$V(P)$ by
$$
d\cdot (v\otimes
e^r)=\left(\hbox{deg}\,(v)-\frac{1}{2}(r\,|\,r)\right)\,v\otimes
e^r,$$ where $\hbox{deg}\,(v)=-\sum_{i=1}^pn_i$.

The number $\hbox{deg}\,(v)-\frac{1}{2}(r\,|\,r)$ is called the
degree of $v\otimes e^r$ and denoted by $\deg\,(v\otimes e^r)$.

\subsection{}
\begin{proposition} The affine Lie algebra $\widetilde{\mathcal  G}$
of type $C_l^{(1)}$ is homomorphic to the Lie algebra $J$ generated by
operators $\a^\vee(n),X^\epsilon_{n}(\alpha), c,
d\;(\alpha\in\dot\Delta, n\in\ZZ)$ on $V(P)=S(\widehat{H}^-)\otimes
\CC[P]$, i.e., there exists a unique Lie algebra homomorphism $\pi$ from $\widetilde{\mathcal  G}$ to the Lie subalgebra $J$ of $End(V(P))$ such that
$$\begin{array}{rcl}
\pi(\gamma^{-1}(\alpha_i)\otimes t^{n})&=& \frac{2}{(\a_i|\a_i)}\alpha_i(n),\\
\pi(e_{\alpha}\otimes {t^{n}})&=&X^\ep_{n}({\alpha}),\\
\pi(c)&=&\id,\\
\pi(d)&=&d,
\end{array}$$
that is, $V(P)$ is a $\widetilde{\mathcal  G}$-module.
\end{proposition}

\section{Some computations needed}
\begin{lemma}
For any $l$,  if $\Lambda_s$ is the basic weight of $C_l^{(1)}$, then we have
\beq
\dim_q(L(\Lambda_s))&=&\dim_q(L(\Lambda_{l-s})), \\
\dim_q(L(\Lambda_s))&=&\prod_{n=1}^\infty\frac{(1-q^{2(l+2)n})(1-q^{2(l+2)n-2-2s})(1-q^{2(l+2)n-2l-2+2s})}{(1-q^{n})}.\eeq
\end{lemma}
For the definition of $\dim_q$, one can refer \cite{K} (see p. 183,
Proposition 10.10).

Define $q$-series
\beq \kappa_q(l,r)&=&\sum_{n\in\ZZ}q^{ln^2-rn},\eeq
for $0<r\leq l$. If $r=l$, then
\beq \kappa_q(l,l)&=&2\prod_{n=1}^\infty\frac{(1-q^{4n})^2}{(1-q^{2n})},\eeq
by Gauss identity
\beq \sum_{n\in\ZZ}q^{2n^2-n}&=&\prod_{n=1}^\infty\frac{(1-q^{2n})^2}{(1-q^n)}.\eeq

Suppose that $V=S(\alpha(-1),\alpha(-2),\cdots)\otimes
\CC[\ZZ\alpha]$ with $(\alpha|\alpha)=2$, then $V$ is an irreducible
$A_1^{(1)}$-module isomorphic to $L(\Lambda_0)$ (one can see [FLM]
for details). The degree of $v=\alpha(-n_1)\cdots\alpha(-n_k)\otimes
e^{n\alpha}\in V$ is defined as $-n_1-\cdots-n_k-n^2$ and weight of
$v$ is $-(n_1+\cdots+n_k+n^2)\delta+n\alpha$. Hence \beq \text{ch}
V&=&e^{\Lambda_0}\frac{1}{\prod_{n=1}^\infty(1-e^{-n\delta})}\sum_{n\in\ZZ}e^{-n^2\delta+n\alpha}.
\eeq Moreover,
\begin{eqnarray}
\text{ch}
L(\Lambda_0)&=&e^{\Lambda_0}\frac{\sum_{n\in\ZZ}e^{-(3n^2+n)\delta+3n\alpha}
-\sum_{n\in\ZZ}e^{-(3n^2+n)\delta-(3n+1)\alpha}}{\prod_{n=1}^\infty(1-e^{-n\delta})(1-e^{-n\delta+\alpha})(1-e^{-(n-1)\delta-\alpha})}.
\end{eqnarray}
If $e^{-\delta}, e^{-\alpha}$ are specialized as $q^l,q^r$, respectively, then $V\cong L(\Lambda_0)$ implies:
\begin{lemma}
If $0<r<l$, then
\begin{eqnarray}
\kappa_q(l,r)&=&\prod_{n=1}^\infty\frac{(1-q^{2ln})(1-q^{4ln-2(l-r)})(1-q^{4l(n-1)+2(l-r)})}{(1-q^{2ln-l-r})(1-q^{2l(n-1)+l+r})}.
\end{eqnarray}
\end{lemma}
\begin{proof}
 This lemma  can easily be proved using the  quintuple product identity (see \cite{CS}).
\end{proof}

\section{The module structure}

\subsection{} Let
$\alpha_0\in H^*$ such that $\{\alpha_0, \alpha_1, \cdots,
\alpha_l\}$ is the simple root system of affine Lie algebra
$\widetilde{\mathcal  G}$ and $ \alpha_0(\alpha_0^\vee)=2$,
$\alpha_0(\alpha_1^\vee)=-2$, $\alpha_0(d)=1$ and
$\alpha_0(\alpha_j^\vee)=\alpha_0(c)=0$ ($2\leq j\leq l$). Then
$\delta=\alpha_0+2\alpha_1+2\alpha_2+\cdots+2\alpha_{l-1}+\alpha_l$
is the primitive imaginary root of $\widetilde{\mathcal  G}$. Let
$\Lambda_i\in H^*$ be such that
$$\Lambda_i(\alpha_j^\vee)=\delta_{ij},\quad \Lambda_i(d)=0\quad (0\leq j\leq l).$$

\begin{lemma} With respect  to the Cartan subalgebra $H$ of
$\widetilde{\mathcal G}$, $V(P)$ has the weight space decomposition
$$V(P)=\sum_{\lambda\in {\hbox{\rm weight}}(V(P))}V(P)_\lambda,$$
and the weight space $V(P)_\lambda$ has a basis $v\otimes
e^r$, where $r\in P, \; v \in S(\dot {\mathcal H}^-)$, and
$$\lambda=\Lambda_0+\left(\deg \,(v)-\frac{1}{2}\,(r\,|\,r)\right)\delta+r,$$
so $\deg \,(v)$ and $r$ are uniquely determined by $\lambda$.
\end{lemma}

\subsection{} The following describes the possible
distribution of the maximal weights of $\widetilde{\mathcal
G}$-module $V(\dot Q)$.

\begin{lemma} For any $\lambda\in P(V(\dot Q))$, we have
$$\lambda\leq \Lambda_{j}-\frac{j}{4}\delta,$$
 for some $j\in\ZZ$, where $0\leq j\leq l$.\end{lemma}

\begin{proof} By Lemma 5.1,
$\lambda=\Lambda_0-\left(\,k+\frac{1}{2}(r\,|\,r)\right)\delta+r$,
where $r=\sum_{i=1}^{l-1}k_i\alpha_i+\frac{k_l}{2}\alpha_l\in{P}$ and
$k\in\frac{1}{2}\,\NN$. At first, we have
\begin{eqnarray*}
\frac12(r\,|\,r)\delta-r&=&\frac14(k_1^2+(k_2-k_1)^2+\cdots(k_{l-1}-k_{l-2})^2+(2k_l-k_{l-1})^2)\delta-\sum_{i=1}^{l-1}k_i\alpha_i-\frac{k_l}{2}\alpha_l\\
&=&\frac14[(k_1^2\delta-2k_1(2\alpha_1+2\alpha_2+\cdots+2\alpha_{l-1}+\alpha_l))\\
&&+((k_2-k_1)^2\delta-2(k_2-k_1)(2\alpha_2+\cdots+2\alpha_{l-1}+\alpha_l))\\
&&\cdots\\
&&+((k_{l-1}-k_{l-2})^2\delta-2(k_{l-1}-k_{l-2})(2\alpha_{l-1}+\alpha_l))\\
&&+(k_l-k_{l-1})^2\delta-2(k_l-k_{l-1})\alpha_l].
\end{eqnarray*}
Suppose that
$$\alpha=2\alpha_{i}+\cdots+2\alpha_{l-1}+\alpha_l<\delta.$$
If $n<0$, then $n^2\delta-2n\a>0$. If $n>1$, then
$$(n^2-1)\delta-2(n-1)\alpha=(n-1)((n+1)\delta-2\alpha)>0.$$
Hence we have
$$n^2\delta-2n\alpha\geq0$$
or
$$n^2\delta-2n\alpha\geq \delta-2\alpha.$$
So
\begin{eqnarray*}
\frac12(r\,|\,r)\delta-r&\geq&\frac
s4\delta-\frac12[(2\alpha_{p_1}+\cdots+2\alpha_{l-1}+\alpha_{l})+\cdots+(2\alpha_{p_s}+\cdots+2\alpha_{l-1}+\alpha_{l})]\\
&\geq&\frac
s4\delta-\frac12[(2\alpha_{1}+\cdots+2\alpha_{l-1}+\alpha_{l})+\cdots+(2\alpha_{s}+\cdots+2\alpha_{l-1}+\alpha_{l})]\\
&=&\frac12(\gamma_s\,|\,\gamma_s)\delta-\gamma_s,
\end{eqnarray*}
for some $s$, where
$$\gamma_s=\alpha_1+2\alpha_2+\cdots+(s-1)\alpha_{s-1}+s(\alpha_s+\cdots+\alpha_{l-1})+\frac{s}{2}\alpha_l\in P,$$
and it clear that $\Lambda_s=\Lambda_0+\gamma_s, (\gamma_s |
\gamma_s)=\frac{s}2$. Then we have
\begin{eqnarray*}
\lambda&=&\Lambda_0-\left(\,k+\frac{1}{2}(r\,|\,r)\right)\delta+r\\
&\leq&\Lambda_0-\frac{1}{2}(r | r)\delta+r\\
&\leq&\Lambda_0-\frac{1}{2}(\gamma_s\,|\, \gamma_s)\delta+\gamma_s\\
&=&\Lambda_s-\frac{s}{4}\delta
\end{eqnarray*}
for some $s \ (0\leq s\leq l)$.
\end{proof}

\begin{remark} By the result above, we know that any highest weight of $V(P)$ belongs
to the set
$$\bigcup_{s=0}^l\left\{\Lambda_s-\frac{s}{4}-\frac{p}{2}\delta \mid p\geq 0, s=0,1,\cdots,l\right\}.$$
More precisely, any highest weight vector has the form $v\otimes e^{\gamma_s}$ for some $s$.
\end{remark}

\begin{theorem}
$V(P)$ has the decomposition
$$V(P)=\bigoplus_{s=0}^l V(P)^{[s]},$$
where $V(P)^{[s]}$ is the sum of those irreducible submodules whose
highest weights $\lambda\leq \Lambda_s-\frac{s}4.$
\end{theorem}

\section{Highest weight vectors}
\subsection{}
Define operators
\beq
S(\alpha,z)&=&\exp\left(\dsum_{n>0}\dfrac{\alpha(-n+\frac{1}{2})}{n-\frac{1}{2}}z^{2n-1}\right)\exp\left(-\dsum_{n>0}\dfrac{\alpha(n-\frac{1}{2})}{n-\frac{1}{2}}z^{-2n+1}\right),
\eeq
with series expansion
\beq S(\alpha,z)&=&\dsum_{n\in\frac{1}{2}\ZZ}S_n(\alpha)z^{-2n}.\eeq

\begin{lemma} For $i=1,\cdots,l-1$, we have
$$\{S_n(\alpha_i),S_m(\alpha_i)\}=S_n(\alpha_i)S_m(\alpha_i)+S_m(\alpha_i)S_n(\alpha_i)=-2\delta_{m+n,0}$$
and
$$S_n(\alpha)=(-1)^{2n}S_n(-\alpha),\; n\in\frac{1}{2}\ZZ.$$
\end{lemma}
\subsection{}
Define $\beta_i=\alpha_i$ for $i=1,\cdots,l-1$ and \beq
\beta_l&=&-\sum_{i=1}^{l-1}\frac{i}{l}\alpha_i,\eeq also let \beq
y_i&=&\sum_{j=i}^l2\beta_j,\quad i=1,\cdots,l.\eeq Define \beq
Z^{[s]}(z)&=&\sum_{i\in\ZZ}Z^{[s]}_{\frac{i}{2}}z^{-i}=\sum_{j=1}^sS(y_j,z)-\sum_{j=s+1}^lS(y_j,z),\eeq
for even $s$. Particularly, $Z^{[l]}(z)=-Z^{[0]}(z).$

\begin{remark}
The operators $Z^{[s]}$ are the same as (or isomorphic to) those
defined by Lepowsky and Wilson in \cite{LW1}, \cite{LW2}, where they
are generating operators of vacuum spaces of standard
$A_1^{(1)}$-modules of level $l$. For more details, one can refer to
those two papers.
\end{remark}

\begin{lemma} For any $n\in\frac12\ZZ$, if $v\otimes e^{\gamma_s}$ is a
highest weight vector and $Z_n^{[s]}v\otimes e^{\gamma_s}$ is not
zero, then $Z_n^{[s]}v\otimes e^{\gamma_s}$ is also a highest
weight vector.\end{lemma}

\begin{proof} At first, we give the proof for $s=0$. For $i<l$, we
have
$$S_n(y_i)+S_n(y_{i+1})=\sum_{j\in\ZZ}S_j(\beta_i)S_{n-j}(\beta_i+2\beta_{i+1}+\cdots+2\beta_l),$$
and $(y_j | \a_i)=0$, $j\not=i,i+1$. Hence
\begin{eqnarray*}
-X_0^\ep(\alpha_i)Z_n^{[0]}(v\otimes1)&=&X_0^\ep(\alpha_i)\left\{\sum_{r\not=i,i+1}S_n(y_r)+\sum_{j\in\ZZ}S_j(\beta_i)S_{n-j}(\beta_i+2\beta_{i+1}+\cdots+2\beta_l)\right\}v\otimes
1\\
&=&Y_{\frac{1}{2}}(\alpha_i)\left\{\sum_{r\not=i,i+1}S_n(y_r)+\sum_{j\in\ZZ}S_j(\beta_i)S_{n-j}(\beta_i+2\beta_{i+1}+\cdots+2\beta_l)\right\}v\otimes
e^{\alpha_i}\\
&=&Y_{\frac12}(\alpha_i)\sum_{r\not=i,i+1}S_n(y_r)v\otimes
e^{\alpha_i}\\
&&+Y_{\frac12}(\alpha_i)\sum_{j\in\ZZ}S_j(\alpha_i)S_{n-j}(\beta_i+2\beta_{i+1}+\cdots+2\beta_l)v\otimes
e^{\alpha_i}\\
&=&\sum_{r\not=i,i+1}S_n(y_r)X_0^\ep(\alpha_i)v\otimes1\\
&&+\sum_{k\in\ZZ}E_k(\alpha_i)S_{\frac12-k}(\alpha_i)\sum_{j\in\ZZ}S_j(\alpha_i)S_{n-j}(\beta_i+2\beta_{i+1}+\cdots+2\beta_l)v\otimes
e^{\alpha_i}\\
&=&-\sum_{j\in\ZZ}S_j(\alpha_i)S_{n-j}(\beta_i+2\beta_{i+1}+\cdots+2\beta_l)X^\ep_0(\alpha_i)v\otimes1=0.
\end{eqnarray*}
Moreover, operators $X_0^\ep(\alpha_l)$ and
$X_1^\ep(-(2\alpha_1+\cdots+2\alpha_{l-1}+\alpha_l))$ commute with
$Z_n^{[0]}$, so $Z_n^{[0]}v\otimes 1$ is still a highest weight
vector.

The proof for $s=l$ is the same as above.

For $Z_n^{[s]}$ with $0<s<l$,
$$S_n(y_i)-S_n(y_{i+1})=\sum_{j\in\ZZ+\frac12}S_j(\beta_i)S_{n-j}(\beta_i+2\beta_{i+1}+\cdots+2\beta_l),$$
then
\begin{eqnarray*}
&&X_0^\ep(\alpha_s)Z_n^{[s]}(v\otimes e^{\gamma_s})\\
&=&X_0^\ep(\alpha_s)\left\{\left(\sum_{r<s}-\sum_{r>s+1}\right)S_n(y_r)+\sum_{j\in\ZZ+\frac12}S_j(\beta_s)S_{n-j}(\beta_s+2\beta_{s+1}+\cdots+2\beta_l)\right\}v\otimes
e^{\gamma_s}\\
&=&Y_1(\alpha_s)\left\{\left(\sum_{r<s}-\sum_{r>s+1}\right)S_n(y_r)+\sum_{j\in\ZZ}S_j(\beta_s)S_{n-j}(\beta_s+2\beta_{s+1}+\cdots+2\beta_l)\right\}v\otimes
e^{\gamma_s+\alpha_i}\\
&=&Y_1(\alpha_i)\left(\sum_{r<s}-\sum_{r>s+1}\right)S_n(y_r)v\otimes
e^{\gamma_s+\alpha_i}\\
&&+Y_1(\alpha_i)\sum_{j\in\ZZ}S_j(\alpha_s)S_{n-j}(\beta_s+2\beta_{s+1}+\cdots+2\beta_l)v\otimes
e^{\gamma_s+\alpha_s}\\
&=&\sum_{r\not=i,i+1}S_n(y_r)X_0^\ep(\alpha_s)v\otimes  e^{\gamma_s}\\
&&+\sum_{k\in\ZZ}E_k(\alpha_i)S_{1-k}(\alpha_i)\sum_{j\in\ZZ+\frac12}S_j(\alpha_s)S_{n-j}(\beta_s+2\beta_{s+1}+\cdots+2\beta_l)v\otimes
e^{\gamma_s+\alpha_s}\\
&=&-\sum_{j\in\ZZ}S_j(\alpha_s)S_{n-j}(\beta_s+2\beta_{s+1}+\cdots+2\beta_l)X^\ep_0(\alpha_s)v\otimes
e^{\gamma_s}=0.
\end{eqnarray*}
For other $X^\ep(\alpha_i)$ and $X_0^\ep(\alpha_l)$,
$X_1^\ep(-(2\alpha_1+\cdots+2\alpha_{l-1}+\alpha_l))$, the proof is
similar to the first case. Then $Z_n^{[s]}v\otimes e^{\gamma_s}$ is also a
highest weight vector.
\end{proof}

For $\lambda=\Lambda_0-\sum_{i=0}^lk_i\a_i$, define
$$\deg \lambda=\sum_{i=0}^lk_i,$$
and
$$V(P)_i=\sum_{\lambda\,: \,\deg \lambda=i}V(P)_\lambda,$$
then
$$V(P)=\sum V(P)_i.$$

The $q$-character $\text{ch}_q$ is a map from $V(P)$ to
$\ZZ[q^{\pm1}]$(to $\ZZ[q^{\pm\frac12}]$ if $l$ is even) defined by
\beq ch_q V(P)&=&\sum\dim V(P)_iq^i.\eeq

Define the highest weight vector space of $V(P)^{[s]}$ as
$\Omega_s\otimes e^{\gamma_s}$. Then we have

\begin{theorem} $\Omega_s$ is generated by
operators $Z^{[s]}_i(i\in\frac12\ZZ_-)$. Moreover, \beq \text{\rm
ch}_q\Omega_s=\prod_{n=1}^{\infty}\frac{(1-q^{l(l+2)n})(1-q^{l[(l+2)n-s-1]})(1-q^{l[(l+2)n-l+s-1]})}{(1-q^{ln})}.\eeq
\end{theorem}

\section{Proof of Theorem 6.4}

Let
\beq\widehat{H}^-&=&\bigoplus_{n\in\ZZ_-}H_S(n+\frac{1}{2}).\eeq

Theorem 6.4 will be proved by the following lemmas.

\begin{lemma}
$S(\widehat{H_S}^-)\otimes 1$ can be generated by operators $Z^{[s]}_n${\rm (}$n\in\frac12\ZZ, s=0,\cdots,l${\rm )} on $1\otimes1$.
\end{lemma}
\begin{proof} At first, by the definition of operators $Z^{[s]}(z)$,
\begin{eqnarray*}
Z^{[1]}_n-Z^{[0]}_n&=&2S(y_1),\\
Z^{[2]}_n-Z^{[1]}_n&=&2S(y_2),\\
&\cdots&\\
Z^{[l-1]}_n-Z^{[l-2]}_n&=&2S(y_{l-1}),\\
\end{eqnarray*}
moreover, for $0<s<l$ and $m\in\ZZ$, $y_s(m+\frac{1}2)$ can be
generated by operators $S_n(y_s)(n\in\frac12\ZZ)$. So
$S(H_S^-)\otimes 1$ can be generated by the $Z^{[s]}_n$'s.
\end{proof}

\begin{lemma}
Suppose that $v\in S(\widehat{H_S}-)$, then $v\otimes e^{\gamma_s}$
is a highest weight vector if and only if for all positive integers
$m$,
$$S_{m-\frac12}(\alpha_i)v\otimes1=0\ (0<i<l,\, i\not=s),\quad  S_{m}(\alpha_s)v\otimes1=0 \ (\hbox{\rm when}\  (\a_s|\a_s)=1).$$
\end{lemma}
\begin{proof}
As we know that $v\otimes e^{\gamma_s}$ is a highest weight vector if and only if
$$X^\epsilon_0(\alpha_i)v\otimes e^{\gamma_s}=X^\epsilon_1(-2\alpha_1-\cdots-2\alpha_{l-1}-\alpha_l)v\otimes e^{\gamma_s}=0, \quad i=1,\cdots,l.$$
For any $v\in S(\widehat{H_S}^-)$, it always holds that $$X^\epsilon_0(\alpha_l)v\otimes e^{\gamma_s}=X^\epsilon_1(-2\alpha_1-\cdots-2\alpha_{l-1}-\alpha_l)v\otimes e^{\gamma_s}=0.$$
Let
$$E(\alpha,z)=E^-(\alpha,z)E^+(\alpha,z)=\sum_{j\in\ZZ}E_j(\alpha)z^{-j},$$
then for  $0<i<l$,
$$X^\epsilon_0(\alpha_i)v\otimes e^{\gamma_s}=\epsilon_{\alpha_i}Y_{\frac12}(\alpha_i)v\otimes e^{\gamma_s+\alpha_i}=\epsilon_{\alpha_i}\sum_{j\in\ZZ}E_{j}(\alpha_i)S_{\frac12-j}(\alpha_i)v\otimes e^{\gamma_s+\alpha_i}$$
for $i\not=s$ and
$$X^\epsilon_0(\alpha_i)v\otimes e^{\gamma_s}=\epsilon_{\alpha_i}Y_{1}(\alpha_i)v\otimes e^{\gamma_s+\alpha_i}=\epsilon_{\alpha_i}\sum_{j\in\ZZ}E_{j}(\alpha_i)S_{1-j}(\alpha_i)v\otimes e^{\gamma_s+\alpha_i}$$
for $i=s$. Thus this lemma holds.
\end{proof}

\begin{lemma}
If $v\in S(\widehat{H_S}^-)$ and for all positive integer $m$,
$$S_{m-\frac12}(\alpha_1)v\otimes1=0,$$
then $v$ belongs to the subspace $W_1$ generated by $Z^{[0]}_{\frac{n}2},Z^{[2]}_{\frac{n}2},\cdots,Z^{[l-1]}_{\frac{n}2},Z^{[l]}_{\frac{n}2}=-Z^{[0]}_{\frac{n}2}(n\in\ZZ)$.
\end{lemma}
\begin{proof}
Notice that
\begin{eqnarray*}
Z^{[0]}_{\frac{n}2}&=&-\sum_{r\not=1,2}S_{\frac{n}2}(y_r)-\sum_{j\in\ZZ}S_j(\beta_1)S_{\frac{n}2-j}(\beta_1+2\beta_{2}+\cdots+2\beta_l)\\
&=&-\sum_{j\in\ZZ}S_j(\alpha_1)S_{\frac{n}2-j}(\beta_1+2\beta_{2}+\cdots+2\beta_l)+\hbox{\rm
terms commuting with} \; S(\alpha_1).
\end{eqnarray*}
\begin{eqnarray*}
Z^{[1]}_{\frac{n}2}&=&\sum_{j\in\ZZ}S_{j+\frac{1}{2}}(\alpha_1)S_{\frac{n}{2}-j-\frac{1}{2}}(\beta_1+2\beta_{2}+\cdots+2\beta_l)+\hbox{\rm
terms commuting with} \; S(\alpha_1).
\end{eqnarray*}
and
\begin{eqnarray*}
Z^{[s]}_{\frac{n}2}&=&\sum_{j\in\ZZ}S_j(\alpha_1)S_{\frac{n}2-j}(\beta_1+2\beta_{2}+\cdots+2\beta_l)+\hbox{\rm
terms commuting with} \; S(\alpha_1),
\end{eqnarray*}
for $s\geq 2$. Since $(\alpha_1|\beta_1+2\beta_{2}+\cdots+2\beta_l)=0$, a homogeneous non-zero vector
$$v=v\otimes1=\sum a_{j_1,\cdots,j_k}Z^{[s_1]}_{j_1}\cdots Z^{[s_k]}_{j_k}\otimes1$$
can be written as
$$\sum b_{i_1,\cdots,i_r}S_{i_1}(\alpha_1)\cdots S_{i_r}(\alpha_1)\otimes1, ({i_1}<\cdots<{i_r}\leq 0)$$
where $b_{i_1,\cdots,i_r}$ is a non-zero polynomial commuting with $S(\alpha_1)$. Then $v\in W_1$ if and only if $i_1,\cdots,i_r\in\ZZ$ for any $b_{i_1,\cdots,i_r}$.
It is easy to show that if $b_{i_1,\cdots,i_r}\otimes1\not=0,$ then
$$S_{-j_1}(\alpha_1)\cdots S_{-j_r}(\alpha_1)v=\hbox{\rm a scalar of \;}b_{j_1,\cdots,j_r}\otimes1\not=0.$$
Condition $S_{m-\frac12}(\alpha_1)v\otimes1=0$ implies all $i_1,\cdots,i_r\in\ZZ$, so $v\in W_1$.
\end{proof}

A similar argument  shows the following two lemmas.
\begin{lemma}
If $v\in S(\widehat{H_S}^-)$ and for all positive integer $m$,
$$S_{m-\frac12}(\alpha_1)v\otimes1=0, S_{m-\frac12}(\alpha_2)v\otimes1=0,$$
then $v$ belongs to the subspace generated by $Z^{[0]}_{\frac{n}2},Z^{[3]}_{\frac{n}2},\cdots,Z^{[l-1]}_{\frac{n}2},Z^{[l]}_{\frac{n}2}=-Z^{[0]}_{\frac{n}2}$.
\end{lemma}

\begin{lemma}
If $v\in S(\widehat{H_S}^-)$ and for all positive integer $m$ and $1<i<l$,
$$S_{m-\frac12}(\alpha_i)v\otimes1=0,$$
then $v$ belongs to the subspace generated by $Z^{[0]}_n$.
\end{lemma}

Similarly to the proof for $s=0$ above, for general $s$, we have

\begin{lemma}
Suppose that $v\in S(\widehat{H_S}^-)$ and $0<s<l$. If
$$S_{m-\frac12}(\alpha_i)v\otimes1=0 (0<i<l, i\not=s),\quad  S_{m}(\alpha_s)v\otimes1=0(\hbox{\rm when} (\a_s|\a_s)=1),$$
 for all positive integer $m$, then $v$ belongs to the subspace generated by $Z^{[s]}_n$.
\end{lemma}

\begin{lemma}
For any $0\leq s\leq l$, the element $1\otimes e^{\gamma_s}$ is a highest weight
vector.
\end{lemma}

\begin{lemma} For odd $l\geq3$,
$\Omega_{s}$ has basis
$$
\left\{Z^{[s]}_{n_1}\cdots
Z^{[s]}_{n_k}\otimes1 \mid n_p\in\frac{1}{2}\ZZ_-, n_{p}\leq
n_{p+1}, n_{p}\leq n_{p+r}-1, n_{k-\sigma(s)}\leq-1\right\}.
$$
For
even $l\geq2$, $\Omega_{s}$ has basis
$$\left\{Z^{[s]}_{n_1}\cdots
Z^{[s]}_{n_k}\otimes1 \mid n_p\in\frac{1}{2}\ZZ_-,
n_{p}-n_{p+r}<-1\Rightarrow\sum_{i=0}^rn_{p+i}\in\ZZ, n_{p}\leq
n_{p+r}-1, n_{k-\sigma(i)}\leq-1, \; \right\},$$here
$r=\frac{l-1}{2}$ if $l$ odd and $r=\frac{l}2$ if $l$ even.
$\sigma(s)=s$ for $s\leq r$, otherwise, $\sigma(s)=r+1-s$.
\end{lemma}
\begin{lemma}
For any $0\leq s\leq l$,
$$\text{\rm ch}_q\Omega_{s}=\prod_{n=1}^{\infty}\frac{(1-q^{l(l+2)n})(1-q^{l[(l+2)n-s-1]})(1-q^{l[(l+2)n-l+s-1]})}{(1-q^{ln})}. $$
\end{lemma}
For Lemmas 7.8 and 7.9, one can refer (\cite{LW1}, Theorem 10.4), (\cite{LW2}, Section 14) and
(\cite{Br}, Section 3).

Lemmas 7.1---7.9 prove Theorem 6.4.

\section{Product-sum identities}
Since \beq V(P)&=&\sum_{s=0}^l\Omega_s\otimes
L(\Lambda_s-\frac{s}{4}\delta),\eeq we have the specialized
character \beq \text{ch}_qV(P)&=&\sum_{s=0}^l \text{ch}_q\Omega_s
\;\text{ch}_q L(\Lambda_s-\frac{s}{4}\delta),\eeq the L.H.S. is \beq
\frac{\sum_{n_1,\cdots,n_l\in\ZZ}q^{\frac12(ln_1^2-n_1+ln_2^2-3n_2\cdots+ln_l^2-(2l-1)n_l)}}{\prod_{n=1}^\infty(1-q^{ln})^{l-1}(1-q^{2ln})}\eeq
which equals \beq
\frac{q^{-\frac{l^2}8}[\kappa_{q^\frac12}(l,1)\kappa_{q^\frac12}(l,3),\cdots,\kappa_{q^\frac12}(l,l-1)]^2}{\prod_{n=1}^\infty(1-q^{ln})^{l-1}(1-q^{2ln})}
&=&q^{-\frac{l^2}{8}}\prod_{n=1}^\infty\frac{(1+q^{n-\frac12})^2}{(1-q^{ln})}
\eeq for even $l$, and equals \beq
\frac{q^{-\frac{l^2-1}8}[\kappa_{q^\frac12}(l,1)\kappa_{q^\frac12}(l,3),\cdots,\kappa_{q^\frac12}(l,l-2)]^2\kappa_{q^\frac12}(l,l)}{\prod_{n=1}^\infty(1-q^{ln})^{l-1}(1-q^{2ln})}
&=&2q^{-\frac{l^2-1}8}\prod_{n=1}^\infty\frac{(1-q^{2ln-l})}{(1-q^{2n-1})^2}
\eeq for odd $l$. Where $\kappa_q$ is defined by Eqs. (4.3) and (4.4).

The R.H.S. is \beq \sum_{s=0}^l
\text{ch}_q\Omega_s \;\text{ch}_q
L(\Lambda_s-\frac{s}{4}\delta)&=&\sum_{s=0}^lq^{\frac{(l-s)s}{2}}\text{ch}_q\Omega_s\;\dim_q
L(\Lambda_s). \eeq Then by the computation of $\Omega_s$ and $\dim_q
L(\Lambda_s)$ before, the proof for our main theorems is finished.

\vskip30pt
\def\refname{

\end{document}